\documentclass[11pt]{article}

\usepackage[margin=1in]{geometry}
\usepackage{amsmath}
\usepackage{amssymb}
\usepackage{amsthm}
\usepackage{mathtools}
\usepackage{url}
\usepackage{hyperref}
\usepackage{centernot}
\usepackage[ruled]{algorithm} 
\usepackage{algorithmic}
\usepackage{enumitem}
\usepackage{scalerel}
\usepackage[T1]{fontenc}%<- \guilsinglright
\usepackage{graphicx}%<- resizebox

\DeclareMathOperator{\dSMZ}{-SMZ}
\DeclareMathOperator{\dSDZ}{-SDZ}
\DeclareMathOperator{\dMZ}{-MZ}
\DeclareMathOperator{\dDZ}{-DZ}
\NewDocumentCommand{\indep}{ O{0.45cm}O{0.1cm} }
{
  \mathrel{%
    \hbox to#1{\xleaders\hbox to#2{\hss \raisebox{0.175ex}{\resizebox{!}{0.5\height}{\resizebox{0.5\width}{1.4\height}{=}}}\hss}\hfill}\text{\guilsinglright}
  }%
}

\newtheorem{theorem}{Theorem}
\newtheorem{observation}[theorem]{Observation}
\newtheorem{lemma}[theorem]{Lemma}
\newtheorem{corollary}[theorem]{Corollary}

\theoremstyle{definition}
\newtheorem*{definition}{Definition}
\newtheorem*{notation}{Notation}
\newtheorem{example}[theorem]{Example}
\newtheorem{remark}[theorem]{Remark}

\numberwithin{equation}{section}
\numberwithin{theorem}{section}

\title{Effective Versions of Strong Measure Zero}

\author{Matthew Rayman\thanks{This research was supported in part by National Science Foundation research grant 1900716.}\\ \\Department of Computer Science \\ Iowa State University \\Ames, IA 50011, USA \\ \texttt{marayman@iastate.edu}}

\date{}

\begin{document}

\maketitle

\begin{abstract}
Effective versions of strong measure zero sets are developed for various levels of complexity and computability. It is shown that the sets can be equivalently defined using a generalization of supermartingales called odds supermartingales, success rates on supermartingales, predictors, and coverings. We show Borel's conjecture that a set has strong measure zero if and only if it is countable holds in the time and space bounded setting. At the level of computability this does not hold. We show the computable level contains sequences at arbitrary levels of the hyperarithmetical hierarchy. This is done by proving a correspondence principle yielding a condition for the sets of computable strong measure zero to agree with the classical sets of strong measure zero.

An algorithmic version of strong measure zero using lower semicomputability is defined. We show that this notion is equivalent to the set of NCR reals studied by Reimann and Slaman, thereby giving new characterizations of this set.

Effective strong packing dimension zero is investigated requiring success with respect to the limit inferior instead of the limit superior. It is proven that every sequence in the corresponding algorithmic class is decidable. At the level of computability, the sets coincide with a notion of weak countability that we define.
\end{abstract}

\section{Introduction}
\label{sec:introduction}

Measure theory is a well studied area of mathematics resulting in a notion of smallness for sets of measure zero. Resource bounded versions of measure zero sets were developed by Lutz \cite{AEHNC}. Hausdorff \cite{haus19} refined measure zero sets in the classical (non effective) setting with dimension. This yields the dimension zero sets which are a subclass of the measure zero sets. An effective version of dimension defined by Lutz \cite{lutz03a} along with effective measure have led to many connections and results in areas including complexity and randomness. For more of an overview see \cite{lutz1993quantitative,lutz2005effective}. In this paper, we continue the refinement by developing effective versions of a subclass of dimension zero sets called strong measure zero.

Borel \cite{Borel1919} defined a set $X\subseteq\mathbb{R}$ to be \emph{strong measure zero} (SMZ) if for every sequence $(\epsilon_n)_{n\in \mathbb{N}}$ of positive reals there is a set $(I_n)_{n\in\mathbb{N}}$ of intervals such that 
\begin{equation}\label{SMZ}
X\subseteq \bigcup_{n\in \mathbb{N}}I_n \text{ and } |I_n|\leq \epsilon_n \text{ for all } n\in\mathbb{N}
\end{equation}
where $|I_n|$ is the length of $I_n$. It is easy to see that every countable set has strong measure zero. Borel conjectured that a set is strong measure zero if and only if it is countable. The work of Sierpiński \cite{Sierpiński1928} and Laver \cite{Laver1983-LAVOTC} proved that this is independent of ZFC. Strong measure zero is easily defined for subsets of Cantor space which we focus on in this paper. For a thorough background on strong measure zero sets see \cite{Bartoszynski1999-BARSTO-3}.

An effective version of strong measure zero was studied by Higuchi and Kihara \cite{higuchi-kihara} where they only require \eqref{SMZ} to hold for computable sequences
of $(\epsilon_n)$ yielding a weaker notion of strong measure zero sets.  In doing so, they developed a characterization of strong measure zero sets using a generalized version of martingales called \emph{odds supermartingales}. 

In this work we investigate requiring the supermartingales themselves to be effective, leading instead to a stronger notion than strong measure zero sets in the classical setting. These supermartingale objects are similar to those used to define effective measure and dimension, but we will show in section 3 and that there are equivalent characterizations using the success rates of standard supermartingales.

Besicovitch \cite{besicovitch1956definition} showed that a set being strong measure zero is equivalent to the set having Hausdorff measure zero with respect to all gauge functions. Hitchcock  \cite{hitchcock2003fractal} has shown that Hausdorff dimension can be defined using predictors. We extend this to work for arbitrary gauge functions resulting in a new characterization of the strong measure zero sets and a useful tool to study them in the effective setting.

A natural effective version of Borel's conjecture with time and space bounded resources along with at the computable level exists. This uses an effective countability definition by Lutz \cite{AEHNC} which are roughly sets of sequences than can be uniformly enumerated within the resource bound. We prove that the effectively countable sets have effective strong measure zero at every level. For the time and space bounded setting it is proven that Borel's conjecture holds. However, it does not hold at the computable level.

We investigate algorithmic strong measure zero occurring at the level of lower semicomputability. The main result is that the resulting class is equivalent to the well studied set of NCR reals defined by Reimann and Slaman \cite{reimann2015measures} containing sequences that are not algorithmically random with respect to any continuous probability measure. We will discuss this connection more in section 5, but note for now that in the classical setting these are not equivalent. NCR has been proven to be countable \cite{reimann2022effective}. We thus have the following implication diagram where $\mathrm{\Delta}_R$ represents a time or space bounded resource class and the dashed arrow implies independence from ZFC.
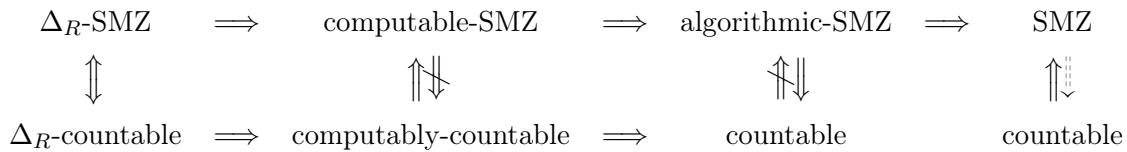
\begin{theorem}
\end{theorem}
\begin{center}\renewcommand{\arraystretch}{1.5}
\newcommand{\ra}{\Longrightarrow}
\newcommand{\uda}{\rotatebox[origin=c]{90}{\(\Longleftrightarrow\)}}
\newcommand{\nra}{\centernot \Longrightarrow}                                                                                                                     
\newcommand{\lra}{\Longleftrightarrow}
\newcommand{\ua}{\rotatebox[origin=c]{90}{\(\Longrightarrow\)}}
\newcommand{\da}{\rotatebox[origin=c]{270}{\(\Longrightarrow\)}}
\newcommand{\nua}{\rotatebox[origin=c]{90}{\(\nra\)}}
\newcommand{\nda}{\rotatebox[origin=c]{270}{\(\nra\)}}
\newcommand{\inddown}{\rotatebox[origin=c]{270}{\(\indep\)}}
\begin{tabular}{ccccccc}
$\mathrm{\Delta}_R$-SMZ        & $\ra$ & computable-SMZ & $\ra$ & algorithmic-SMZ        & $\ra$ & SMZ\\
$\uda$         &        & $\ua \nda$  &        & $ \nua\da$         &        &$\ua \inddown$\\
$\mathrm{\Delta}_R$-countable & $\ra$ & computably-countable & $\ra$ & countable & & countable\\
\end{tabular}
\end{center}

We show that effective strong measure zero sets can be defined using effective coverings resulting in a characterization similar to Borel's original definition \eqref{SMZ}. This results in another relatively simple characterization of NCR. While there are analagous classical results as well as in Higuchi and Kihara's version, our setting requires use of different techniques. Using this result, we prove a correspondence principle giving sufficient conditions for computable, algorithmic, and classical versions of strong measure zero to agree on a set. This result is analogous to Hitchcock's \cite{hitchcock2005correspondence} correspondence principle for effective dimension. We use our result to show that there are sequences at arbitrarily high levels of the hyperarithmetic hierarchy with computable strong measure zero, mirroring a result for NCR.

Many of the characterizations of strong measure zero sets involve success with respect to a limit superior. We therefore investigate what happens if success is required in the limit inferior. For dimension, Lutz \cite{lutz03a} showed a limit superior condition on martingales characterizes Hausdorff dimension while Althreya, Hitchcock, Lutz and Mayordomo \cite{AHLM07} showed packing dimension can be characterized with a limit inferior. Requiring limit inferior success at all gauges leads to sets with \emph{strong packing dimension zero} (SPDZ). In the classical setting, the existence of uncountable strong packing dimension zero sets is also independent of ZFC, see for example \cite{zindulka2012small}. For time and space restrictions the sets are the same as strong measure zero. However at the computable level, the sets now correspond with a weak computably-countable notion that we define. It is also shown that  the algorithmic version contains exactly the decidable sequences, giving the following implication chart.

\begin{theorem}
\end{theorem}
\begin{center}\renewcommand{\arraystretch}{1.5}
\newcommand{\ra}{\Longrightarrow}
\newcommand{\uda}{\rotatebox[origin=c]{90}{\(\Longleftrightarrow\)}}
\newcommand{\nra}{\centernot \Longrightarrow}
\newcommand{\lra}{\Longleftrightarrow}
\newcommand{\ua}{\rotatebox[origin=c]{90}{\(\Longrightarrow\)}}
\newcommand{\da}{\rotatebox[origin=c]{270}{\(\Longrightarrow\)}}
\newcommand{\nua}{\rotatebox[origin=c]{90}{\(\nra\)}}
\newcommand{\nda}{\rotatebox[origin=c]{270}{\(\nra\)}}
\newcommand{\inddown}{\rotatebox[origin=c]{270}{\(\indep\)}}
\begin{tabular}{ccccccc}
$\mathrm{\Delta}_R$-SPDZ        & $\ra$ & computable-SPDZ & $\ra$ & algorithmic-SPDZ        & $\ra$ & SPDZ\\
$\uda$         &        & $\uda$  &        & $\uda$         &        &$\ua \inddown$\\
$\mathrm{\Delta}_R$-countable & $\ra$ & weak computably-countable & $\ra$ & $\subseteq$DEC & $\ra$ & countable\\
\end{tabular}
\end{center}

\section{Preliminaries}
We work in the \emph{Cantor space} $\mathbf{C}=\{0,1\}^\infty$ consisting of all infinite binary sequences. A \emph{string} is a finite, binary string $w\in\{0,1\}^*$ which has length $|w|$. We write $\lambda$ for the empty string of length 0. For $0\leq i \leq j \leq |w|-1$ we write $w[i\dots j]$ for the string consisting of the $i^{\text{th}}$ through $j^{\text{th}}$ bits of $w$. Similarly for $i,j\in \mathbb{N}$ with $i\leq j$, we let $S[i\dots j]$ be the string consisting of the $i^{\text{th}}$ through $j^{\text{th}}$ bits of $S$.

For a string $x$ and a string or sequence $y$ we let $xy$ be the concatenation of $x$ and $y$. We say that $x\sqsubseteq y$ if there is a string or sequence $z$ such that $y=xz$. We say $x\sqsubset y$ if $x\sqsubseteq y$ and $x\not = y$. 

Fix an injective function $\langle\enspace,\enspace \rangle$ from $\{0,1\}^*\times \{0,1\}^*$ onto $\{0,1\}^*$ such that $\langle x,y\rangle$ and the projections $\langle x,y \rangle\to x,\langle x, y \rangle \to y$ are all computable in polynomial time.

For a string $x$, the \emph{cylinder} at $x$ is 
\[C_x=\{S\in \mathbf{C} \mid x \sqsubseteq S\}.\]

We associate a language $A\subseteq \{0,1\}^*$ with its \emph{characteristic sequence} ${\chi_A\in \mathbf{C}}$ defined by
\[
    \chi_A(i)= 
\begin{cases}
    1,& \text{if } s_i\in A\\
    0,              & \text{otherwise}
\end{cases}
\]
where $s_0,s_1,\dots$ is the enumeration of all strings in $\{0,1\}^*$ in standard lexicographic order.

\begin{definition}
    A \emph{time bounded resource class} is a set ${T\subseteq\{f:\{0,1\}^*\to \{0,1\}^*\}}$ for which there is a sequence $(g_n)_{n\in \mathbb{N}}$  of functions $g_n:\mathbb{N}\to \mathbb{N}$ satisfying
\begin{enumerate}
    \item $g_n=o(g_{n+1})$ for all $n\in \mathbb{N}.$
    \item There is some computable $h:\mathbb{N}\to \mathbb{N}$ such that $g_n=o(h)$ for all $n\in \mathbb{N}$.
    \item $T=\bigcup_{n\in \mathbb{N}}\{f:\{0,1\}^*\to \{0,1\}^*\mid f\in \mathrm{DTIME}(g_n)\}$.
\end{enumerate} 
\emph{Space bounded resource classes} are defined analogously with DSPACE. In this paper we will let $\mathrm{\Gamma}_R$ be the set consisting of all time or space bounded resource classes that contain either P or PSPACE. 
\end{definition}

We also will look at the classes
\begin{align}
\mathrm{all} &= \{f:\{0,1\}^*\to \{0,1\}^*\}\nonumber \\
\mathrm{comp} & =  \{f\in \mathrm{all} \mid f \text{ is computable} \}\nonumber
\end{align}
and use $\mathrm{\Gamma}_{RCA}$ to be the set $\mathrm{\Gamma}_R$ along with the classes comp and all. Note that the class all corresponds to classical results.

\begin{definition}[Lutz \cite{lutz1990category}]
    A \emph{constructor} is a function $\delta: \{0,1\}^* \to \{0,1\}^*$ such that $x\sqsubset \delta(x)$ holds for all $x\in \{0,1\}^*$. 
\end{definition}
\begin{definition}[Lutz \cite{lutz1990category}]
    The \emph{result} of a constructor $\delta$ is the sequence ${R(\delta)\in \mathbf{C}}$ such that \\$\delta^i(\lambda)\sqsubseteq R(\delta)$ for all $i\in \mathbb{N}$.
\end{definition}
For $\mathrm{\Delta}\in \mathrm{\Gamma}_{RCA}$, let 
\[R(\mathrm{\Delta})=\{R(\delta) \mid \delta\in \mathrm{\Delta} \text{ is a constructor}\}.\]
We have for example \cite{lutz1990category}
\begin{align}
R(\mathrm{all}) &= \mathbf{C} \nonumber\\
R(\mathrm{comp}) & =  \mathrm{DEC}\nonumber\\
R(\mathrm{p}) & = \mathrm{E} \nonumber \\
R(\mathrm{p}\mathrm{space}) & =   \text{E}\text{SPACE}\nonumber
\end{align}
where $\mathrm{p}$ is polynomial time resource bounded, $\mathrm{psace}$ is polynomial space bounded, $\mathrm{DEC}$ is the set of decidable languages and $\mathrm{E},\mathrm{ESPACE}$ are the standard complexity classes.

\begin{definition}
    For a discrete domain $D$ such as $\{0,1\}^*$ we say a function ${f:D\to[0,\infty)}$ is \emph{lower semicomputable} if there is a computable function $g:\mathbb{N}\times D\to \mathbb{Q}\cap[0,\infty)$, called a \emph{computation} of $f$, such that for all $ r\in\mathbb{N},x\in D$,
    \[g(r,x)\leq g(r+1,x)\leq f(x)\]
    and 
    \[\lim_{r\to\infty}g(r,x)=f(x).\]
\end{definition}

Similarly, for $\mathrm{\Delta}\in \mathrm{\Gamma}_{RCA}$ and a discrete domain $D$ we say that $f:D\to [0,\infty)$ is $\mathrm{\Delta}$-computable if there is an function $g:\mathbb{N}\times D\to \mathbb{Q}\cap [0,\infty)\in \mathrm{\Delta}$ called a \emph{computation} of $f$ such that ${|g(r,x)-f(x)|\leq 2^{-r}}$ for all $r\in \mathbb{N}$. In this paper $\mathrm{\Gamma}$ denotes $\mathrm{\Gamma}_{RCA}$ together with the lower semicomputable class. We will also refer to the lower semicomputable level as $\emph{algorithmic}$.

We also need a basic understanding of functionals and their complexity. Functionals generalize functions and can have have inputs and outputs that are functions themselves. Functionals in this paper will mostly have the form of 
\[F:(A \to B) \to (D \to \mathbb{R})\] 
where $A,B$ and $D$ are all discrete domains. Therefore on an input function $f:(A\to B)$ the functional outputs a function $F(f)=g$ for some $g:D\to \mathbb{R}$. To talk about the complexity and computability we look at the functional
\[F':(A \to B) \to (\mathbb{N}\times D \to \mathbb{Q})\]
where $F'(f)=g$ is a computation of $F(f)$.
For $\mathrm{\Delta}\in\mathrm{\Gamma}$ equal to comp or lower semicomputable, we say that the functional $F$ is in $\mathrm{\Delta}$ if there is an oracle Turing machine $M$ such that for all $f:A \to B,r\in \mathbb{N}$ and $x\in D$,  $M$ on input $(r,x)$  with oracle access to $f$ outputs $g(r,x)$ where $g\in\mathrm{\Delta}$ is a computation of $F(f)$ as defined above. 

For time bounded $\mathrm{\Delta}$ we require that the time $M$ takes to output $g(x,r)$ is $h(|x|+r+\Bar{g}(h(|x|+r))$ for some $h$ in $\mathrm{\Delta}$ where $\Bar{g}(n)$ is the maximum length of the oracle $g$'s output on a string of length at most $n$. For example if $\mathrm{\Delta}$ is polynomial time then $M$ can query polynomially far out in the input length and use the maximum length of the query results as a parameter for the polynomial running time. Different notions of functional complexity exist, but all the results in the paper hold for every natural definition and most oracles will not impact the complexity in our use cases.

We therefore use $\mathrm{\Gamma}_R, \mathrm{\Gamma}_{RCA}$ and $\mathrm{\Gamma}$ to represent classes of functions or functionals which will be clear from context. 

\section{Effective Strong Measure Zero Characterizations}
\subsection{Strong Measure Zero}
We begin by utilizing the characterization of strong measure zero developed by Higuchi and Kihara.

\begin{definition}[Higuchi and Kihara \cite{higuchi-kihara}]
    An \emph{odds} function is any function $O:\{0,1\}^*\rightarrow[1,\infty)$. $O$ is said to be \emph{acceptable} if $\prod_{w\sqsubset S}O(w)=\infty$ for all $S\in \mathbf{C}$.
\end{definition}

Higuchi and Kiharas definition of an $O$-supermartingale below can be viewed as a generalization of the other following martingale objects we will use in the paper. 
\begin{definition}\label{martingales}
    Let  $O:\{0,1\}^*\rightarrow[1,\infty)$ be an odds function and $s\in[0,\infty)$.
    \begin{enumerate}
        \item An \emph{$O$-supermartingale} is a function ${d:\{0,1\}^*\rightarrow[0,\infty)}$ which satisfies
        \begin{equation}
            d(w)\geq \frac{d(w0)}{O(w0)}+\frac{d(w1)}{O(w1)} \label{smc}
        \end{equation}
        for all $w\in \{0,1\}^*$.
        \item  An \emph{$O$-martingale} is an $O$-supermartingale that satisfies \eqref{smc} with equality for all $w\in \{0,1\}^*$.
        \item An \emph{s-supergale} is an $O$-supermartingale with the constant odds function $O(w)=2^s$ for all $w\in \{0,1\}^*$.
        \item An \emph{s-gale} is an $O$-martingale with the constant odds function $O(w)=2^s$ for all $w\in \{0,1\}^*$.
        \item A supermartingale is a 1-supergale.
        \item A martingale is a 1-gale.
    \end{enumerate}
\end{definition}

We will use the following easy to verify result throughout the paper.

\begin{observation}
    Let $(d_n)_{n\in\mathbb{N}}$ be a sequence of $O$-supermartingales such that $\sum_{n\in\mathbb{N}}d_n(\lambda)< \infty$. Then $d=\sum_{n\in\mathbb{N}}d_n$ is a $O$-supermartingale. The same holds for the other martingale objects in Definition \ref{martingales}.
\end{observation}

\begin{definition}
    Let $X\subseteq \mathbf{C}$ and $d$ be one of the martingale objects in Definition \ref{martingales}. We say that $d$ \emph{succeeds} on $X$ if 
    \[\limsup_{n\to \infty} d(S[0\dots n])=\infty\]
    for all $S\in X$.
\end{definition}

\begin{theorem}[Higuchi and Kihara \cite{higuchi-kihara}]
    A set $X\subseteq \mathbf{C}$ has strong measure zero if and only if, for every acceptable odds $O:\{0,1\}^*\to[1,\infty)$, there is an $O$-supermartingale that succeeds on $X$.
\end{theorem}

We will also make use of the following standard definitions and a lemma in their paper.
\begin{definition}
    An $\emph{outer premeasure}$ is a monotone subadditive atomless function $\mu:\{0,1\}^*\to [0,\infty)$. That is, for all $w\in \{0,1\}^*, a\in\{0,1\}$ and $S\in\mathbf{C}$ the following three properties hold: 
    \begin{enumerate}
        \item (monotone) $\mu(w)\geq \mu(wa)$.
        \item (subadditive) $\mu(w)\leq \mu(w0)+\mu(w1)$.
        \item (atomless) $\liminf_{n\to \infty}\mu(S[0\dots n])=0$.
    \end{enumerate}
    
    $\mu$ can be extended to an \emph{outer measure} $\mu^*:\mathcal{P}(\mathbf{C})\to [0,\infty)$ by the "Method I construction"\cite{rogers1998hausdorff} to obtain
    \begin{equation}\label{outerpremeasure}
        \mu^*(X)=\inf\{\sum_{w\in A}\mu(w) \mid A\subseteq \{0,1\}^* \text{ and } X\subseteq \bigcup_{w\in A}C_w\}.
    \end{equation}
\end{definition}

\begin{lemma}[Higuchi and Kihara \cite{higuchi-kihara}]\label{premeasure}
    Let $O$ be an odds function with $O(w0)^{-1}+O(w1)^{-1}\geq 1$ for all $w\in\{0,1\}^*$. Then the function $\mu_O:\{0,1\}^*\to [0,\infty)$ defined by 
    \[\mu_O(w)=\prod_{i=0}^{|w|-1}O(w[0\dots i])^{-1}\]
    is an outer premeasure.
\end{lemma}

One can then show that there is an $O$-supermartingale that succeeds on a set $X$ if and only if $\mu_O^*(X)=0$. In their paper, they defined a set as having effective strong measure zero if for every computable odds function $O$ there exists an $O$-supermartingale that succeeds on the set. Here we instead use the following definition forcing effectivity on the $O$-supermartingales themselves.

\begin{definition} 
    Let $\mathrm{\Delta}\in \mathrm{\Gamma}$. A set $X\subseteq \mathbf{C}$ has \emph{$\mathrm{\Delta}$-strong measure zero}, $X\in\mathrm{\Delta}\dSMZ$, if there is a functional
    \[F:(\{0,1\}^*\to \mathbb{Q}\cap[1,2]) \rightarrow (\{0,1\}^*\to[0,\infty))\]  
    in $\Delta$ such that for every acceptable odds function ${O:\{0,1\}^*\to\mathbb{Q}\cap[1,2]}$, $F(O)$ is an\\ $O$-supermartingale that succeeds on $X$. We say $F$ is an odds functional and write $F^O$ for $F(O)$.
\end{definition}

 The fact that we limit odds functions to rationals in the interval $[1,2]$ may appear to be a weaker condition, but we will see that the definition is robust to this change and others later in this section. Note that all our odds functions therefore correspond to outer premeasures as in Lemma \ref{premeasure}. From now on we will assume all odds functions are acceptable unless stated otherwise. We can also now see the connection between effective strong measure zero, measure zero, and dimension zero.

\begin{definition}[Lutz  \cite{AEHNC,lutz03a,lutz2003dimensions}] Let $X\subseteq\mathbf{C}$ and $\mathrm{\Delta}\in \mathrm{\Gamma}$
    \begin{enumerate}
        \item $X$ has $\mathrm{\Delta}$-\emph{measure zero} ($X\in\mathrm{\Delta}\dMZ$), if there is a supermartingale $d\in \mathrm{\Delta}$ that succeeds on $X$.
        \item $X$ has $\mathrm{\Delta}$-\emph{dimension zero} ($X\in\mathrm{\Delta}\dDZ$), if there is a $s$-supergale $d\in \mathrm{\Delta}$ that succeeds on $X$ for all $s\in\mathbb{Q}\cap(0,1)$.
    \end{enumerate}
\end{definition}

We therefore have the desired connections.
\begin{observation}
    For $X\subseteq \mathbf{C}$ and $\mathrm{\Delta}\in \mathrm{\Gamma}$ the following holds:
    \[X\in\mathrm{\Delta}\dSMZ \implies X\in\mathrm{\Delta}\dDZ \implies X\in\mathrm{\Delta}\dMZ\]
\end{observation}

\subsection{Strong Dimension Zero}
In this section we show that strong measure zero can be equivalently defined in terms of success rates of supermartingales as is the case for dimension \cite{lutz2003dimensions,downey2010algorithmic}. 
\begin{definition}
    A \emph{gauge function}  is a function $g:\mathbb{N}\to \mathbb{Q}^+$ that is non-increasing with\\ ${\lim_{n\to\infty}g(n)=0}$.
\end{definition}

\begin{definition}
    A supermartingale $d$ \emph{$g$-succeeds} on $X\subseteq\mathbf{C}$ for a  gauge function $g$ if 
\[\limsup_{n\to \infty}\frac{d(S[0\dots n-1])}{2^{n}g(n)}=\infty\]
for all $S\in X$.
\end{definition}

\begin{remark}\label{gauges}
    Staiger \cite{staiger2017exact} showed that a set has Hausdorff measure zero with respect to $g$ if and only if there is a supermartingale $d$ that $g$-succeeds on it. In his work he looked at determining an \emph{exact} gauge function for a set at the lower semicomputable and computable level. In this work we will instead be looking at sets that succeed on all gauges with oracle access to the gauges. In Staiger's paper, and typically in fractal geometry, gauge functions have domain and range of $\mathbb{R}^+$, see for example \cite{falc14} for more background. The domain corresponds to diameters of sets which in our setting of Cantor space  will be of the form $2^{-n}$ for $n\in \mathbb{N}$. Schnorr \cite{DBLP:journals/mst/Schnorr71} also looked at success rates with respect to order functions that can be seen as inverses of computable gauge functions.
\end{remark}

\begin{definition}
Let $\mathrm{\Delta}\in \mathrm{\Gamma}$. A set $X\subseteq \mathbf{C}$ has \emph{$\mathrm{\Delta}$-strong dimension zero}, ${X\in \Delta\dSDZ}$, if there is a functional  
\[F:(\mathbb{N}\rightarrow\mathbb{Q}^+) \rightarrow (\{0,1\}^*\rightarrow[0,\infty))\]
in  $\mathrm{\Delta}$ such that for every gauge function $g:\mathbb{N}\rightarrow[0,\infty)$, $F(g)$ is a supermartingale that $g$-succeeds on $X$. We say $F$ is a dimension functional and write $F^g$ for $F(g)$.
\end{definition}

\begin{lemma}{\label{OddsSMtoSM}}
    Let $O:\{0,1\}^*\to [1,\infty)$ be an acceptable odds function. Then a function \\${d:\{0,1\}^*\to [0,\infty)}$ is a (super)martingale if and only if ${d':\{0,1\}^*\to [0,\infty)}$ defined by
    \[d'(w)=\frac{d(w)}{2^{|w|}\mu_O(w)}\]
    is an $O$-(super)martingale.
\end{lemma}
\begin{proof}
Let $d$ and $d'$ be as given. Then
  \begin{align}
 &d(w)\geq \frac{d(w0)+d(w1)}{2}\nonumber \\
\iff  &d(w)\geq \frac{d'(w0)2^{|w0|}\mu(w0)+d'(w1)2^{|w1|}\mu(w1)}{2}\nonumber\\
\iff  &d(w)\geq \frac{2^{|w|}\mu(w)(d'(w0)O(w0)^{-1}+d'(w1)O(w1)^{-1})}{1}\nonumber\\
\iff  &d'(w)\geq \frac{d'(w0)}{O(w0)}+\frac{d'(w1)}{O(w1)}.\nonumber
\end{align} 
For the martingale version equality holds in every line.
\end{proof}

We will later see that following lemma holds for all $\mathrm{Delta}\in\mathrm{\Gamma}$ (see remark \ref{robust}).
\begin{lemma}\label{SMZ-SP}
    For $\mathrm{\Delta} \in \mathrm{\Gamma}$ allowing at least exponential time or polynomial space resources, a set $X$ has $\mathrm{\Delta}\dSMZ$ if and only if it has $\mathrm{\Delta}\dSDZ$.
\end{lemma}
\begin{proof}
Suppose $X$ has $\mathrm{\Delta}$-strong measure zero with witnessing odds functional $F$. We will create a gauge functional $G$ as follows. Given a gauge function $g$, without loss of generality suppose $g(n)\leq 2g(n+1)$ for all $n$, otherwise $g$ can be transformed within $\mathrm{\Delta}$'s restrictions into a harder $g'$ satisfying this with $g'(n)\geq g(n)$ for all $n$. Now let $G$ use $F$ with odds function $O$ where $O(w)=\frac{g(|w|-1)}{g(|w|)}$ for all strings $w$ with $|w|>0$ and $g(-1)=1$. Then we have $\mu_O(w)=g(|w|)$ and
\[G^g(w)=F^O(w)2^{|w|}g(|w|)\] 
is a supermartingale by Lemma \ref{OddsSMtoSM}. Moreover, 
\[\limsup_{n\to \infty} \frac{G^g(S[0\dots n])}{2^ng(n)}=\limsup_{n\to \infty}F^O(S[0\dots n])=\infty\] 
for all $S\in X$.

For the other direction suppose $X$ has $\mathrm{\Delta}$-strong dimension zero with witnessing gauge functional $G$. Let $F$ be an odds functional that given a odds function $O$ computes a gauge function
\[g(n)=\max\{\mu_O(w)\mid w\in \{0,1\}^n\}.\]
Note that by compactness and the fact that $O$ is acceptable we have $lim_{n\to \infty} g(n)=0$. Therefore,
\[F^O(w)=\frac{G^g(w)}{2^{|w|}\mu_O(w)}\] 
is an $O$-supermartingale by Lemma \ref{OddsSMtoSM} and 
\[\limsup_{n\to \infty} F^O(S[0\dots n])\geq\limsup_{n\to \infty}\frac{G^g(S[0\dots n])}{2^ng(n)}=\infty\]
for all $S\in X$.
These conversions in both directions can be done within the restrictions of $\mathrm{\Delta}$.
\end{proof}

\begin{observation}\label{oddsrange}
    Let $D$ be any of the following sets or its intersection with $\mathbb{Q}$. Then, defining $\mathrm{\Delta}$ strong measure zero as the sets succeeding on all acceptable odds functions $O:\{0,1\}^*\to D$ results in equivalent definitions for every $\mathrm{\Delta}\in \mathrm{\Gamma}$.
    \begin{enumerate}
        \item $[1,2]$
        \item $[1,\infty)$
        \item $\{1,2\}$
        \item $(1,2]$
    \end{enumerate}
\end{observation}
\begin{proof}
    Let $A$ and $B$ any of these two sets and suppose that a set $X\subseteq\mathbf{C}$ has strong measure zero in terms of odds functions $O:\{0,1\}^*\to A$ with $F$ being a witnessing odds functional. We will show there is an odds functional $G$ with odds functions $O:\{0,1\}^*\to B$ that succeeds on $X$. Given an acceptable odds function $O:\{0,1\}^*\to B$, compute an acceptable odds function $O':\{0,1\}^*\to A$ such that 
    \begin{equation}\label{leq}
        \mu_{O'}(w)\leq c\mu_O(w)
    \end{equation}
    for some fixed constant $c$ and all $w\in\{0,1\}^*$. Then let $G$ be the odds functional defined by 
    \[G^O(w)=\frac{F^O(w)\mu_O(w)}{\mu_{O'}(w)}.\]
    By Lemma \ref{OddsSMtoSM} we have that $G^O$ is an $O$-supermartingale and by \eqref{leq} ${G^O(w)\geq \frac{1}{c}F^O(w)}$ so it succeeds on $X$. 

    For an example of how to construct $O'$, let $A=(1,2]\cap \mathbb{Q}$ and $B=[1,\infty)$. Then given odds function $O:\{0,1\}^*\to B$ define $O':\{0,1\}^*\to A$ recursively in the following way. If $O(w)>2$, let $O'(w)=2$ If $O(w)=1+a_w\in(1,2)$, let $O'(w)\in \mathbb{Q}$ be such that $1+\frac{a_w}{2}\leq O'(w)\leq 1+a_w$. If $O(w)=1$, let $O'(wx)=1+\frac{1}{2^{|x|+1}}$ for all $x\in\{0,1\}^*$ until $\prod_{i=0}^{|x|-1}O(wx[0\dots i])\geq 3$. It is routine to check that $O'$ is an acceptable odds function that satisfies \eqref{leq} with $c=3$.
    
\end{proof}
Similarly, it is easy to see restricting gauge functions to be strictly decreasing would result in an equivalent characterization by combining the above two results.

\subsection{Strong Predictability}
In this section we generalize Hitchcock's \cite{hitchcock2003fractal} connection of predictors and dimension. This results in new characterizations of gauged dimension and strong measure zero. For "harsh" odds functions, a martingale will have to bet arbitrarily close to all of its money on a single bit, and in the effective case most examples will bet all the money on a bit, making it easy to talk in terms of predicting bits with certainty. We first generalize predictors to superpredictors allowing a connection to supermartingales.
\begin{definition}
    A \emph{superpredictor} is a function $\pi:\{0,1\}^* \times \{0,1\}\to [0,1]$ such that 
    \begin{equation}\label{super predictor}
    \pi(w,0)+\pi(w,1)\leq1
    \end{equation}
    for all $w\in \{0,1\}^*$
    A \emph{predictor} is a superpredictor that has equality for \eqref{super predictor}
    for all $w\in \Sigma^*$.
\end{definition}
The value $\pi(w,a)$ is $\pi$'s prediction that the next bit in a sequence starting with $w$ will be $a$. The following is routine to verify.

\begin{notation}
    Given a (super)predictor $\pi$, the function $d_\pi$ given by the recursion
    \[d_\pi(\lambda)=1\]
    \[d_\pi(wa)=2d_\pi(w)\pi(w,a)\]
    is a (super)martingale. Conversely, For a (super)martingale $d$, the function $\pi_d$ given by 
    \[\pi_d(w,a)=
    \begin{cases}
        \frac{2d(wa)}{d(w)},& \text{if } d(w)>0\\
        \frac{1}{2},& \text{if } d(w)=0
    \end{cases}
    \]
    is a (super)predictor.
\end{notation}

Given $p\in[0,1]$ we let $loss^{\log}(p)=\log\frac{1}{p}$ be the loss associated with a superpredictor predicting the next bit correctly with probability $p$. Therefore it has no loss if it predicts correctly with probability $1$ and scales logarithmically to infinity the closer $p$ is to zero. The cumulative loss below is equal to the sum of the loss on each bit in a string.

\begin{definition}
    The $\emph{log-cumulative loss}$ of a superpredictor $\pi$ on a string ${w\in \{0,1\}^*}$ is 
    \[\mathcal{L}^{\log}_{\pi}(w) =-\log(\prod_{i=0}^{|w|-1}\pi(w[0\dots i-1],w[i])).\]
\end{definition}

\begin{definition}
    A \emph{prediction order} is a function $h:\mathbb{N}\to \mathbb{Q}^+$ that is nondecreasing and unbounded. We say a predictor $\pi$ \emph{$h$-succeeds} on $X\subseteq\mathbf{C}$ if 
    \[\liminf_{n\to \infty}\frac{\mathcal{L}^{\log}_\pi(S[0\dots n-1])}{h(n)}<1\]
for all $S\in X$.
\end{definition}

\begin{definition}
Let $\mathrm{\Delta}\in\mathrm{\Gamma}$. A set $X\subseteq \mathbf{C}$ is \emph{$\mathrm{\Delta}$-strongly predictable} if there is a functional  
\[P:(\mathbb{N}\rightarrow\mathbb{Q}^+) \rightarrow (\{0,1\}^*\times \{0,1\}\rightarrow[0,1])\]
 in $\mathrm{\Delta}$ such that for every prediction order $h$, $P(h)$ is a predictor that $h$-succeeds on $X$.
\end{definition}

\begin{lemma}\label{predictordimension}
    A (super)martingale $d$ $g$-succeeds on a set $X$ for a gauge function $g$ if and only if the (super)predictor $\pi_d$ $h$-succeeds on $X$ for 
    \[h(n)=-\log g(n).\]
\end{lemma}
\begin{proof}
    Let $d$ and $\pi_d$ be as given. Note that we have 
    \[\mathcal{L}^{\log}_{\pi_d}(w)=-\log\prod_{i=0}^{|w|-1}\pi_d(w[0\dots i-1],w[i])=-\log(2^{-|w|}d(w))=|w|-\log d(w).\] 
    We therefore have
    \begin{align}
    &\limsup_{n\to \infty}\frac{d(S[0\dots n])}{2^ng(n)}=\infty\nonumber  \\
    \iff &\liminf_{n\to \infty}\frac{2^ng(n)}{d(S[0\dots n])}=0 \nonumber \\
    \iff &\liminf_{n\to \infty}n+\log g(n)-\log d(S[0\dots n])=-\infty \ \nonumber \\
    \iff &\liminf_{n\to \infty}\mathcal{L}^{\log}_{\pi_d}(S[0\dots n])+\log g(n)=-\infty \ \nonumber \\
    \iff &\liminf_{n\to \infty}\frac{\mathcal{L}^{\log}_{\pi_d}(S[0\dots n])}{-\log g(n)}<1 \ \nonumber
    \end{align} 
    noting that the last equivalence is true since $\log g(n)$ is nonincreasing and goes to $-\infty$ while $\mathcal{L}^{\log}_{\pi_d}(S[0\dots n])$ is positive.
\end{proof}

Combining this with Staiger's result mentioned in Remark \ref{gauges} we have the following.

\begin{corollary}
    Let $\mathrm{\Delta}\in \mathrm{\Gamma}$. A set $X\subseteq \mathbf{C}$ has $\Delta$-Hausdorff measure zero with respect to a gauge $g$ if and only if there is a $\mathrm{\Delta}$-computable superpredictor that $-\log(g)$-succeeds on $X$.
\end{corollary}

We therefore have the following which we will again see later is true for all $\mathrm{\Delta}\in \mathrm{\Gamma}$.
\begin{theorem}\label{equiv}
    For a set $X\subseteq\mathbf{C}$ and  $\mathrm{\Delta}\in \Gamma$ allowing exponential time or polynomial space the following are equivalent:
    \begin{enumerate}
        \item $X$ has $\mathrm{\Delta}$-strong measure zero.
        \item $X$ has $\mathrm{\Delta}$-strong dimension zero.
        \item $X$ is $\mathrm{\Delta}$-strongly predictable.
    \end{enumerate}
\end{theorem}
\begin{proof}
    we have $1 \iff 2$ by Lemma \ref{SMZ-SP}.  $2 \iff 3$ follows from Lemma \ref{predictordimension} as a predictor functional can use a dimension functional to succeed on the same set and vice versa.
\end{proof}

\section{Effective Borel Conjecture}
In this section we investigate Borel's conjecture at the $\mathrm{\Gamma}_{RCA}$ levels of effectivity where there is the following natural definition of countability.

\begin{definition}[Lutz \cite{AEHNC}]
    Let  $\mathrm{\Delta}\in \mathrm{\Gamma}_{RCA}$. A set $X\subseteq \mathbf{C}$ is \emph{$\mathrm{\Delta}$-countable} if there is a function $\delta:\mathbb{N} \times \{0,1\}^*\rightarrow \{0,1\}^*$ with the following properties.
    \begin{enumerate}
        \item $\delta \in \mathrm{\Delta}$.
        \item For each $k\in \mathbb{N}$, if we write $\delta_k(w)=\delta(k,w)$, then the function $\delta_k$ is a constructor.
        \item $X\subseteq \{R(\delta_k) \mid k\in \mathbb{N}\}$.
    \end{enumerate}
\end{definition}
\begin{remark}
    In the original definition it was required that $X= \{R(\delta_k) \mid k\in \mathbb{N}\}$. For our purposes, we want all subsets of countable sets to be countable. Otherwise it is easy to come up with a set of languages $L_i$ with $|L_i|=1$ whose union is not $\mathrm{\Delta}$-countable while having $\mathrm{\Delta}$-strong measure zero.
\end{remark}
\begin{lemma}\label{ctblsmz}
    Let $\mathrm{\Delta}\in \mathrm{\Gamma}$. If $X$ is $\Delta$-countable then $X\in \mathrm{\Delta}\dSMZ$.
\end{lemma}
\begin{proof}
    Let $\delta$ be a witness to $X$ being $\Delta$-countable and define an odds functional $F$ as follows. Given odds function $O$, for each $i\in \mathbb{N}$ let $d_i$ be an $O$-martingale where $d_i(\lambda)= \frac{1}{2^i}$ and 
    \[d_i(wa)=
        \begin{cases}
        d_i(w)\cdot O(wa),& \text{if } wa\sqsubseteq R(\delta_i)\\
        0, & \text{otherwise} 
        \end{cases}
\]
for all $w\in\{0,1\}^*$ and $a\in\{0,1\}$. Now let $d=\sum_{i\in \mathbb{N}}d_i$. Then $F(O)=d$ is a $O$-martingale and $d_i$ ensures that 
\[\limsup_{n\to \infty}d(R(\delta_i)[0\dots n])\geq\frac{1}{2^i} \limsup_{n\to \infty}\prod_{j=0}^{n-1}O(R(\delta_i)[0\dots j])=\infty.\]
To see that $d$ is $\mathrm{\Delta}$-computable let $d':\mathbb{N}\times \{0,1\}^* \to [0,\infty)\in\mathrm{\Delta}$ be the function 
\[d'(r,w)=\sum_{i=0}^{|w|+r}d_i(w).\]
We then have 

\begin{align}
    d(w)-d'(r,w) &=\sum_{i=|w|+r+1}^\infty d_i(w)\nonumber  \\
     &\leq 2^{|w|}\sum_{i=|w|+r+1}^\infty \frac{1}{2^i} \nonumber \\
     &=2^{-r}. \nonumber 
    \end{align} 
\end{proof}

While it is known that strong measure zero implying countability is independent of ZFC in the classical case, we will see the answers in the effective versions depend on the level of effectivity.

\subsection{Time and Space Bounded Strong Measure Zero}

We begin with the following result for singleton sets consisting of a sequence.
\begin{lemma}\label{nonresult}
    Let $f(n)$ be a function and $S\in\mathbf{C}$ be a sequence with $S\not= R(\delta)$ for every \\${\delta\in \mathrm{DTIME}(f(n))}$. Then there is an odds function $O$ such that no odds functional\\ $F\in \mathrm{DTIME}(f(n))$ succeeds on $S$. Similarly for DSPACE.
\end{lemma}
\begin{proof}
    We will prove the DTIME version. Let $F_0,F_1,\dots$ be an enumeration of the odds functionals in $\mathrm{DTIME}(f(n))$. We construct $O$ in steps $s\in \mathbb{N}$ with $O_s$ being the odds function after step $s$ and the starting $O_{-1}$ being the constant function $O(w)=1$ for all $w\in\{0,1\}^*$. Note that this is not an acceptable odds function, but the resulting $O$ will be. We first describe the construction and then prove it works.

    \begin{algorithm}
    \caption{}
    \begin{algorithmic}[1]
    \STATE Let $n_{-1}=0,m_{-1}=0$, and  $s=0$
    \STATE Let $n_s>m_{s-1}$ be such that $F_i^{O_{s-1}}(S[0\dots n_s])\leq \frac{1}{2}$ for each $0\leq i \leq s$
    \STATE Let $m_s>m_{s-1}\in \mathbb{N}$ be the minimum length such that for all $0\leq i \leq s$, $F_i$ does not query $O_{s-1}$ on any string of length $m_s$ on any input $w$ with $|w|\leq n_s$ and precision parameter $r=1$ in the computation.
    \STATE Set $O_s(w)=2$ for all $w$ such that $|w|=m_s$ and $O_s(w)=O_{s-1}(w)$ for all other $w$
    \STATE Set $s=s+1$ and go to step 2
    \end{algorithmic}
    \end{algorithm}

    To see that this is possible at each stage, note that the oracle $O_{s-1}$ contains a finite amount of information. Moreover,  as $F_i$ is resource bounded it has to eventually produce output even though $O_{s-1}$ is eventually all $1$'s and not acceptable. for every $i\leq s$ there must be infinitely many $k\in \mathbb{N}$ where $F_i^{O_{s-1}}(S[0\dots k+1])\leq \frac{1}{2}F_i^{O_{s-1}}(S[0\dots k])$ as otherwise a function could be created in $\mathrm{DTIME}(f(n))$ that computes $S$. There are also only $s-1$ possible naturals $k$ such that $F_i^{O_{s-1}}(S[0\dots k+1])> F_i^{O_{s-1}}(S[0\dots k])$. Hence ${F_i^{O_{s-1}}(S[0\dots n])\leq \frac{1}{2}}$ for all sufficiently large $n$.
    
    Thus, any odds functional $F_i$ will have 
    \[F_i^O(S[0\dots n]) < 2\]
    for all $n\geq n_i$ and hence not succeed on $S$.
\end{proof}

To generalize this we prove the following result about $\mathrm{\Delta}$-countable sets.

\begin{lemma}\label{countability}
    Let $\mathrm{\Delta}\in \mathrm{\Gamma}_R$. Then there is a sequence of sets $X_0,X_1,X_2\subseteq \mathbf{C}$ such that the following hold.
    \begin{enumerate}
        \item Each $X_i$ is $\mathrm{\Delta}$-countable.
        \item $X_i\subseteq X_{i+1}$ for each $i\in \mathbb{N}$.
        \item for every $S\in R(\mathrm{\Delta})$, there is an $i$ such that $S\in X_i$.
    \end{enumerate}
\end{lemma}
\begin{proof}
    We prove it for time bounds, the proof for space bounds is analogous. Let $f_0,f_1,f_2, \dots$ be a sequence of functions such that each $f_i:\mathbb{N}\to \mathbb{N}\in \mathrm{\Delta}$, $f_i=o(f_{i+1})$, and $\mathrm{\Delta}=\bigcup_{i\in \mathbb{N}} \mathrm{DTIME}(f_i)$.
    
    Now for each $f_i$ create a function $\delta_i:\mathbb{N}\times \{0,1\}^* \to \{0,1\}^*$ such that ${\delta_{i,k}(w)=\delta_i(k,w)}$ does the following. Let $k=\langle j,c\rangle$ for some $j,c\in \mathbb{N}$. Then run the Turing machine $M_j$ on input $s_{|w|}$ for $f_i(|w|)+c$ steps. If it halts and accepts then output $w1$, otherwise $w0$. Then each $\delta_{i,k}$ is a constructor and $\{R(\delta) \mid \delta \in \mathrm{DTIME}(f_i)\}=\{R(\delta_{i,k})\mid k\in \mathbb{N}\}$. Hence, letting $X_i=\{R(\delta_{i,k})\mid k\in \mathbb{N}\}$ the lemma holds.
\end{proof}

We are now able to prove the main result of this section.

\begin{theorem}\label{resourceSMZ}
    If a set $X$ has $\mathrm{\Delta}\dSMZ$ for some  $\mathrm{\Delta}\in \mathrm{\Gamma}_R$ then $X$ is $\mathrm{\Delta}$-countable.
\end{theorem}
\begin{proof}
    We prove the contrapositive, suppose $X$ is not $\mathrm{\Delta}$-countable and Let $f_0,f_1,f_2,\dots$ be a sequence of functions that define $\mathrm{\Delta}$. By Lemma \ref{countability}, for every $m\in\mathbb{N}$ there must be some $S\in X$ such that $S\neq R(\delta)$ for every  $\delta\in \mathrm{DTIME}(f_m)$. Thus, by Lemma \ref{nonresult} there is an $O$ that no odds functional in $\mathrm{DTIME}(f_m)$ succeeds on for $S$.
\end{proof}

\begin{remark}\label{robust}
    The proofs in this section can be translated over to gauge functions and predictors in a natural way resulting in sets having $\mathrm{\Delta}$-strong dimension zero and being $\mathrm{\Delta}$-strongly predictable if and only if they are $\Delta$-countable. In particular, Theorem \ref{equiv} holds for all $\mathrm{\Delta}\in \mathrm{\Gamma}$.
\end{remark}

The hypothesis of NP not having measure zero in E and the weaker hypothesis of NP not having dimension zero in E have led to many interesting consequences. For specific definitions and examples see \cite{lutz1993quantitative,lutz2023dimension}. For the case of strong measure zero, we get an even weaker hypothesis that gives a tight bound on the complexity of NP unlike in the other two cases. 

\begin{corollary}
    $\mathrm{NP}$  has strong measure zero in $\mathrm{E}$ if and only if \[\mathrm{NP}\cap\mathrm{E}\subseteq\mathrm{DTIME}(2^{kn})\]
    for some fixed $k\in \mathbb{N}$. 
\end{corollary}

\subsection{Computable Strong Measure Zero}
At the level of computability, we will see that strong measure zero does not imply countability.  We start by defining a class of languages that are in a certain sense as close as possible to being computable.

\begin{definition} 
    $S\in \mathbf{C}$ is \emph{almost constructible} if there exists a computable 
    \[\delta:\mathbb{N} \times \{0,1\}^* \rightarrow \{0,1\} \times \{0,1\}^*\] 
    such that for every $ w \sqsubseteq S$ and $n\in \mathbb{N}$, $\delta(n,w)=(b,x)$ where $|x|=n$ and either $wb\sqsubseteq S$ or $wx\sqsubseteq S$.
\end{definition}

\begin{example}
    Let  $M_0,M_1,M_2\dots$ be an enumeration of Turing machines and
    \[f(n)=\max\{t\in \mathbb{N} \mid \exists k\in \mathbb{N}, k\leq n \text{ and } M_k(k)\text{ halts in exactly } t \text{ steps}\}.\]
    Then the language $A=\mathrm{Graph}(f)=\{\langle n,f(n)\rangle \mid n\in \mathbb{N}\}$ is almost constructible, but not decidable.
\end{example}

\begin{lemma}
    If $S$ is almost constructible then $\{S\}$ has computable strong measure zero.
\end{lemma}
\begin{proof}
    Let $S$ be as given with $\delta$ being a witness to $S$ being almost constructible. Define a computable predictor functional that works as follows on prediction order $h$. By Observation \ref{oddsrange} we can assume $h$ is strictly increasing. The following recursive algorithm initiated with $w=\lambda$ and $L=0$ will give the values of a superpredictor $\pi$.
    
\begin{algorithm}[H]
\caption{Predictor}
\begin{algorithmic}[1]
\STATE Input $w, L$:
\STATE Compute a $p\in (0,1)$ such that $L+\log\frac{1}{p}<h(|w|+1)$ 
\STATE Compute $n\in \mathbb{N}$ such that $L+\log\frac{1}{1-p}< h(|w|+n)$
\STATE Compute $\delta(w,n)=(b,x)$
\STATE Set $\pi(w,b)=p$ and $\pi(w,1-b)=1-p$ where $1-b$ is the opposite bit of $b$
\STATE Set $\pi(w(x[0\dots i-1]), x[i]))=1$ for all $0< i< n$
\STATE Recursively start Predictor on inputs ($wb, L+\log\frac{1}{p}$) ($wx, L+\log\frac{1}{1-p}$)
\end{algorithmic}
\end{algorithm}
To compute the value $\pi(x,a)$ the above algorithm can be run and eventually the value is determined or all the active $w$'s will be longer than $x$ in which case 0 can be output making $\pi$ is a superpredictor.
Steps 2 and 3 ensure that $\pi$ $h$-succeeds on $S$.
\end{proof}

Since the singleton set $\{S\}$ is not computably countable for every undecidable $S$, we have the following.
\begin{corollary}
    There exists an $X\subseteq \mathbf{C}$ with computable strong measure zero that is not computably countable.
\end{corollary}

We will come back later in section 6 to investigate the complexity of sequences $S$ where $\{S\}$ has computable strong measure zero. For such $S$ we say $S$ has computable strong measure zero.

\section{Algorithmic Strong Measure Zero}
We now investigate strong measure zero for functionals that are lower semicomputable. We first show that there is an optimal algorithmic dimension functional.

\begin{definition}
    A \emph{continuous semimeasure} on $\mathbf{C}$ is a function $\mu:C\rightarrow \mathbb{R}^+$ such that $\mu(\lambda)\leq 1$ and $\mu(w)\geq \mu(w1)+\mu(w0)$ for all $w\in\{0,1\}^*$. 
\end{definition}
\begin{theorem}[Levin \cite{zvonkin1970complexity}]
    There is a universal lower semicomputable continuous semimeasure $\mathbf{M}$. That is, for every lower semicomputable continuous semimeasure $\mu$ there is a constant $c>0$ such that 
    \[\mathbf{M}(w)\geq c\mu(w)\]
    for all $w\in\{0,1\}^*.$
\end{theorem}

\begin{theorem}[Schnorr \cite{DBLP:journals/mst/Schnorr71}]
    The function $\mathbf{d}:\{0,1\}^*\rightarrow \mathbb{R}$ defined by 
    \[\mathbf{d}(x)=2^{|x|}\mathbf{M}(x)\] 
    is a universal lower semicomputable supermartingale. That is, for every lower semicomputable supermartingale f there is a constant $c>0$ with 
    \[\mathbf{d}(w)\geq cf(w)\]
    for all $w\in\{0,1\}^*.$
\end{theorem}

Both of these theorems relativize to any oracle giving us the following. 

\begin{corollary}
    There is an universal algorithmic functional $\mathbf{F}$ defined by ${\mathbf{F}(g)=\mathbf{d}^g}$ for all gauge functions $g:\mathbb{N}\to \mathbb{Q}^+$. Specifically, a set $X\subseteq\mathbf{C}$ has algorithmic strong measure zero if and only if for every gauge function $g$ and every $S\in X$,
    \[\limsup_{n\to \infty} \frac{\mathbf{d}^g(S[0\dots n])}{2^{n}g(n)}=\infty.\]
\end{corollary}

Unlike the other levels of effectivity, a set $X$ has algorithmic strong measure zero if and only if $\{S\}$ has algorithmic strong measure zero for every $S\in X$. We will therefore focus on individual sequences and say that a sequence has algorithmic strong measure zero if $\{S\}$ has algorithmic strong measure zero.

We will make use of relativized \emph{a priori complexity} 
\[\mathrm{KM}^g(x)=\log\frac{1}{\mathbf{M}^g(x)}\]
in order to classify the sequences with algorithmic strong measure zero. 

\begin{lemma}\label{algsmz}
    A sequence $S$ has algorithmic strong measure zero if and only if for every gauge function $g$ there are infinitely many $n$ such that 
    \begin{equation}
        \mathrm{KM}^g(S[0\dots n])\leq \log\frac{1}{g(n)}.\label{lemmacondition}
    \end{equation}
\end{lemma}
\begin{proof}
    Let $S$ be any sequence and $g$ be a gauge function. We have that $S$ has algorithmic strong measure zero if and only if
      \begin{align}
&\limsup_{n \to \infty}\frac{\mathbf{d}^g(S[0\dots n])}{2^{n}g(n)}=\infty\nonumber \\
\iff  &\limsup_{n \to \infty}\frac{\mathbf{M}^g(S[0\dots n])}{g(n)}=\infty\nonumber\\
\iff  &\limsup_{n \to \infty}\frac{2^{-\mathrm{KM}^g(S[0\dots n])}}{g(n)}=\infty. \label{toprove}
\end{align}
Note that ${\mathrm{KM}^g(x)=\mathrm{KM}^{g(n)^2}(x)+O(1)}$. Hence, if we assume that $\eqref{lemmacondition}$ is true then there is a $c\in \mathbb{R}$ and infinitely many $n$ with 
\[\mathrm{KM}^g(S[0\dots n])\leq \log\frac{1}{g(n)^2}+c.\]
For each such $n$ we have 
\[2^{-\mathrm{KM}^g(S[0\dots n])}\geq 2^{-c}g(n)^2\]
and therefore
\[\frac{2^{-\mathrm{KM}^g(S[0\dots n])}}{g(n)^2}\geq 2^{-c}.\]
Since $\lim_{n\to \infty}\frac{g(n)}{g(n)^2}=\infty$, $\eqref{toprove}$ is true.

For the other direction we will prove the contrapositive. Suppose there is gauge $g$ such that $\mathrm{KM}^g(S[0\dots n])> \log\frac{1}{g(n)}$ for all sufficiently large $n$. Then we have 
\[\frac{2^{-\mathrm{KM}^g(S[0\dots n])}}{g(n)}\leq 1\]
for all sufficiently large $n$ so $\eqref{toprove}$ does not hold.
\end{proof}

Reimann and Slaman have defined \cite{reimann2015measures} a class of sequences called \emph{never continuously random}, denoted NCR, consisting of all sequences that are never effectively random with respect to any continuous probability measure. Here continuous means the same as atomless, but the outer measures for strong measure zero in $\eqref{premeasure}$ are more general than probability measures. The following is an overview of the definition of NCR, but see the paper for the details on their representations of measures.

\begin{definition}[Reimann and Slaman \cite{reimann2015measures}]
    Let $\mu$ be a probability measure and $r_\mu$ be a representation of $\mu$. Then
    \begin{itemize}
        \item A Martin-Löf-$\mu$ test relative to $r_\mu$ is a sequence of uniformly $\mathrm{\Sigma}_1^0$ sets $(W_n)_{n\in \mathbb{N}}$ relative to $r_\mu$ with $\mu(W_n)\leq 2^{-n}$ for all n.
        \item $X\in \mathbf{C}$ passes a Martin-Löf-$\mu$ test relative to $r_\mu$ if $X\not \in \bigcap_{n\in \mathbb{N}} W_n$.
        \item $X \in \mathbf{C}$ is $\mu$-random if it passes every Martin-Löf-$\mu$ test relative to $r_\mu$ for some representation $r_\mu$ of $\mu$.
        \item $X\in \mathbf{C}$ is in NCR if it its not $\mu$-random for every continuous probability measure $\mu$.
    \end{itemize}
\end{definition}

In a related paper Reimann defined the following notion of complexity extending a definition of complex sequences by Kjos-Hanssen, Merkle, and Stephan \cite{kjos2006kolmogorov}.

\begin{definition}[Reimann \cite{reimann2008effectively}]
    For a unbounded increasing function $h:\mathbb{N}\rightarrow\mathbb{N}$, a sequence $S$ is \emph{strongly $h$-complex relative to $Z$} if $h$ is computable from $Z$ and 
    \[\mathrm{KM}^Z(S[0\dots n])>h(n)\]
    for all $n$.
\end{definition}
Based off the work of Reimann \cite{reimann2008effectively}, Li showed the following.
\begin{lemma}[\cite{li2020algorithmic,reimann2008effectively}]
    A sequence $S$ is in NCR if and only if for every $Z\in\mathbf{C}$ and unbounded increasing function $h:\mathbb{N}\rightarrow\mathbb{N}$ computable from $Z$ with $h(n+1)\leq h(n)+1$ for all $n$, $S$ is not strongly $h$-complex relative to $Z$.
\end{lemma}

\begin{theorem}\label{NCR}
    A sequence $S$ has algorithmic strong measure zero if and only if it is in NCR.
\end{theorem}
\begin{proof}
    First assume $S$ has algorithmic strong measure zero and consider any $Z$. Then for every unbounded increasing $h$ computable from $Z$ we have that ${g(n)=2^{-h(n)}}$ is a gauge also computable from $Z$. As $S$ has algorithmic strong measure zero, by Lemma \ref{algsmz} there is a constant $c$ and infinitely many n with
     \[\mathrm{KM}^Z(S[0\dots n])\leq \log\frac{1}{g(n)^{\frac{1}{2}}}
     +c<h(n).\]

     For the other direction let $S$ be in NCR. Then given a gauge function $g$ we can computably transform it into a gauge function $g'$ with $g'(n)\geq g(n)$ and such that $h(n)=log(\frac{1}{g'(n)})$ is an unbounded increasing function with $h(n+1)\leq h(n)+1$ for all $n$. Then $S$ is not $h$ complex relative to $g$ so there are infinitely many $n$ such that 
     \[\mathrm{KM}^g(S[0\dots n])\leq h(n).\]
     Note that if there were only finitely many $n$ then an $h'$ could be computed by changing finitely many values so that S is strongly $h'$-complex. Thus,
     there are infinitely many $n$ with 
    \[\mathrm{KM}^g(S[0\dots n])\leq \log\frac{1}{g'(n)}\leq \log\frac{1}{g(n)}.\]
    Hence, by Lemma \ref{algsmz} $S$ has algorithmic strong measure zero.
     
\end{proof}
    In the classical setting, the NCR class corresponds to sets that are measure zero with respect to all Borel atomless probability measures. These sets are referred to as universal measure zero which is a weaker property than strong measure zero. In fact, the existence of uncountable sets with universal measure zero is provable in ZFC \cite{miller1984special}. However, Theorem \ref{NCR} says that at the algorithmic level these two notions are equivalent. Reimann and Slaman also investigated $\mathrm{NCR}_n$ random sequences allowing more complex Martin-Löf tests. They were able to show each of these are countable, but in order to do so they needed an application of Borel determinacy and hence the existence of infinitely many iterates of the power set of the natural numbers \cite{reimann2022effective}.

\section{Effective coverings and a Correspondence Principle}
In this section we will look at how effective coverings can be used to characterize strong measure zero similar to Borel's original definition.

\begin{definition}
    Let  $\mathrm{\Delta}\in \mathrm{\Gamma}_{RCA}$, $A\subseteq \mathbb{N}$ and $X\subseteq \mathbf{C}$. A function ${f^A:\mathbb{N}\to \{0,1\}^*}$ is a \emph{$\mathrm{\Delta}$-$A$ covering of $X$} if $f\in \mathrm{\Delta}$ with oracle $A:\mathbb{N}\to \{0,1\}$ satisfies the following.
    \begin{enumerate}
        \item If $n\in A$ then $f^A(n)=w$ for some $w$ with $|w|= n$. 
        \item For every $S\in X$ there exits an $n\in A$ such that $f^A(n)\sqsubseteq S$.
    \end{enumerate}
    We refer to $\{f^A(n)\mid n\in A\}$ of such an $f$ as a \emph{$\mathrm{\Delta}$-$A$ cover} of $X$ 
    and say that $X$ is \emph{strongly $\mathrm{\Delta}$-coverable} if there is a functional $F\in \mathrm{\Delta}$ such that $F(A)$ is a $\mathrm{\Delta}$-$A$ covering of $X$ for all infinite $A\subseteq\mathbb{N}$.

    At the algorithmic level $\mathrm{\Delta}$, the above definitions hold for $f^A$ being a partial recursive function that can also be undefined in condition 1 above. Specifically, $f$ corresponds to a oracle Turing machine that outputs a string of length $n$ or does not halt on every $n\in A$.
\end{definition}

\begin{lemma}
    For $\mathrm{\Delta}\in\mathrm{\Gamma}$, a set $X$ has $\mathrm{\Delta}$-strong measure zero if and only if it is strongly $\mathrm{\Delta}$-coverable.
\end{lemma}

\begin{proof}
    It is easy to see that this is true for $\mathrm{\Delta}\in \mathrm{\Gamma}_R$ as well as $\mathrm{\Delta}=\mathrm{all}$ where it corresponds to the normal classical definition. We will prove that it is true for $\mathrm{\Delta}$ equal to computable or lower semicomputable.
    First, suppose $X$ has algorithmic (computable) strong measure zero and let $F$ be a odds functional witnessing this. Without loss of generality suppose $F^O(\lambda)< 1$ for all odds functions $O$. We will define a Turing machine $M$ that works as follows on oracle ${A=\{a_0,a_1,a_2,\dots\}\subseteq \mathbb{N}}$ in standard order. Let $f_A:\mathbb{N}\rightarrow\mathbb{N}$ be the function 
    \[f_A(n)=a_{2^{n+2}-3}\]
    and consider the odds function
    \[
    O(w)= 
\begin{cases}
    2,& \text{if } |w|\in \mathrm{Range}(f_A)\\
    1,              & \text{otherwise}
\end{cases}
\]
Then let $M^A$ perform the following algorithm.

    \begin{algorithm}[H]
    \caption{}
    \begin{algorithmic}[1]
    \STATE On input $a_i\in A$:
    \STATE Compute the minimum $n\in \mathbb{N}$ such that $i\leq 2^{n+2}-3$.
    \STATE Let $c= i-2^{n+1}+3$.
    \STATE Run $F^O$ and for each $w$ with $|w|=f_A(n)$ where $F^O(w)\geq 1$ is found set $c=c-1$.
    \STATE Output $w[0\dots a_i]$ when $c$ gets set to zero in step 4. If $\mathrm{\Delta}$=comp then output $0^{a_i}$ if every $w$ with $|w|=f_A(n)$ has been checked and $c>0$.
    \end{algorithmic}
    \end{algorithm}
    
Note that by construction the maximum number of strings $w$ with ${|w|=f_A(n)}$ and $F^O(w)>=1$ is at most $2^n$ in order for the supermartingale condition $\eqref{smc}$ to be satisfied. By construction, a prefix of each of these $w$ is output for every $n$. Moreover, for every $S\in X$, $F^O(S[0...n])$ can only increase on the lengths in $\mathrm{Range}(f_A)$. Therefore $M^A$ will output a prefix of $S$ for every $S$ with $\limsup_{n\to \infty} F^O(S[0\dots n])>1$ and hence all $S\in X.$

For the other direction, let $M^A$ be an oracle Turing machine that witnesses $X$ being strongly $\mathrm{\Delta}$-coverable. We will define a odds functional that succeeds on $X$ given odds function $O$ as follows. Let 
\[f_O(n)=\min\{m\in \mathbb{N}\mid \prod_{i=0}^{m-1} O(w[0\dots i])\geq n \text{ for all } w\in \{0,1\}^m\}\]
be the minimum length $m$ so that the product of odds along every path of length $m$ is at least $n$, noting it is possible by compactness. We will create an infinite set $d_0,d_1,d_2,\dots$ of $O$-martingales. 

For $i\in\mathbb{N}$ let 
\[A_i=\{f_O(2^{i+n})\mid n\in \mathbb{N}\}.\] 
Then let $d_i$ be an $O$-martingale defined as follows.
\begin{itemize}
    \item Initially let $d_i(w)=0$ for all $w\in \{0,1\}^*$
    \item For each $m=f_O(2^{i+n})\in A_i$ for some $n\in \mathbb{N}$ such that $M^{A_i}(m)=w$:
    \begin{itemize}
    
        \item Increase $d_i(\lambda)$ by $2^{-(i+n)}$. 
        \item Increase $d_i(w[0\dots j])$ by $2^{-(i+n)}\prod_{k=0}^j O(w[0\dots k])$ for $j< |w|$.
        \item Increase $d_i(w0^m)$ by $2^{-(i+n)}\prod_{j=0}^{ |w|-1} O(w[0\dots j])\prod_{k=0}^mO(w0^k)$ for all $m\in \mathbb{N}$.
    \end{itemize}
\end{itemize}
 Then $d_i$ is a $\Delta$-computable $O$-martingale. For each $w$ in the range of $M^{A_i}$ we have
\[d_i(w)\geq2^{-(i+n)}2^{2i+n}=2^i\]
and hence $d_i$ shows that the 
\[\limsup_{n\to \infty}d^O(S[0\dots n])\geq 2^i\]
for all $S\in X$.
Now letting
\[F(O)=d=\sum_{i\in \mathbb{N}} d_i\]
we have that $F(O)$ succeeds on all $S\in X$. It is clear that $F\in \mathrm{\Delta}$ and we have $d(\lambda)<\infty$ since
\[d(\lambda)=\sum_{i\in \mathbb{N}}d_i(\lambda)\leq \sum_{i\in \mathbb{N}}(\sum_{n\in \mathbb{N}}2^{-(i+n)})=\sum_{i\in \mathbb{N}}2^{-i+1}=4.\]
\end{proof}

In the above proof we used a functional giving coverings to create a odds functional that outputs odds martingales for all odds functions. Therefore, odds martingales can be used instead of odds supermartingales to define $\mathrm{\Delta}$-strong measure zero. It  follows that dropping the super on other characterizations also makes no impact. This is similar to effective dimension where it was originally defined with supergales before Hitchcock \cite{hitchcock2003gales} proved gales suffice.

    \begin{corollary}
        Let $\mathrm{\Delta}\in \mathrm{\Gamma}$. Then for every  $X\subseteq \mathbf{C}$ there is an odds functional $F$ succeeding on $X$ if and only if there is an odds functional $F'$ succeeding on $X$ where $F'(O)$ is an odds martingale for all odds functions $O$. Moreover, there is a universal algorithmic martingale functional.
    \end{corollary}
    
We will now use the covering characterization to investigate the complexity of sequences with computable strong measure zero. To do so, we start by proving a correspondence theorem for computable and algorithmic strong measure zero that is an analogue to the result of Hitchcock \cite{hitchcock2005correspondence} for computable and algorithmic dimension. We will use the following characterization of $\mathrm{\Pi}_1^0$ and $\mathrm{\Sigma}_2^0$ definable classes.

\begin{definition}
Let $X\subseteq \mathbf{C}$.
\begin{itemize}
    \item $X\in \mathrm{\Pi}_1^0$ if there is a computable tree $T\subseteq\{0,1\}^*$ whose infinite paths are the sequences $S\in X$.
    
    \item $X\in \mathrm{\Sigma}_2^0$ if there is a computable function $f:\mathbb{N}\times \{0,1\}^*\to \{0,1\}$ where ${T_i=\{w \mid f(i,w)=1\}}$ are computable trees and ${X=\bigcup_{i\in \mathbb{N}}X_i}$ is the union of the $\mathrm{\Pi}_1^0$ sets corresponding to the $T_i$'s.
\end{itemize}

\end{definition}

\begin{lemma}\label{closed covers}
    If a $\mathrm{\Pi}_1^0$ class $X$ has an $A$ covering then it has a computable $A$ covering.
\end{lemma}
\begin{proof}
    Since $X$ is a closed subset of a compact space, it is itself compact and hence there is a finite cover of $A$. Let $f$ be the computable function that on each input dovetails in a fixed manner over all  potential finite $A$ covers until it finds one that succeeds and outputs accordingly.
    Note that verifying a covering is semidecidable since there can only be finitely many strings in $T$ that are not inside the cover by Kőnig's lemma.
\end{proof}

\begin{lemma}\label{correspondence}
     Let $X$ be a $\mathrm{\Sigma}_2^0$ set. Then the following are equivalent.
   \begin{enumerate}
       \item $X$ is strongly coverable
       \item $X$ is algorithmically strongly coverable
       \item $X$ is computably strongly coverable
   \end{enumerate}
   Moreover 1 and 2 are equivalent for every union of $\mathrm{\Pi}_1^0$ sets.
\end{lemma}
\begin{proof}
It is clear that $3 \implies 2 \implies 1$. To see that $1 \implies 3$ let $F$ be the computable functional that given $A$ first finds an $A$ covering of the first $\mathrm{\Pi}_1^0$ set in $X$ as in Lemma \ref{closed covers}. Then it removes the finitely many naturals in the covering to get $A'$ and continues until the input $n$ is in a cover.  For the second part, note that any union of algorithmic strong measure zero sets has algorithmic strong measure zero.
\end{proof}

Using this we can see that the computable strong measure zero sets contain sequences that are as complex as the algorithmic strong measure zero sequences in a particular sense by using the following result.

\begin{theorem}[Cenzer et al. \cite{cenzer1986members}]\label{jump}
   Let $\alpha$ be a computable ordinal. Then there is a countable $\mathrm{\Pi}_1^0$ class containing $x$ with $x\equiv_\mathrm{T}\emptyset^{(\alpha)}.$
\end{theorem}

Combining Theorem \ref{jump} and Lemma \ref{correspondence} we have the following.
\begin{corollary}
   for every computable ordinal $\alpha$ there is a $x$ with  $x\equiv_\mathrm{T}\emptyset^{(\alpha)}$ and $\{x\}$ having computable strong measure zero.
\end{corollary}
Kjos-Hanssen and Montalbán \cite{AntonioThesis} had proved this result for NCR using Theorem \ref{jump} as well. Reimann and Slaman \cite{reimann2015measures} showed that this is the best that can be done by proving if $x$ is not hyperarithmetic then it is not in NCR.

\section{Strong Packing Dimension Zero}
In this section we look at what happens if success is required in the limit inferior. We will use the following definition, but it is equivalent to our other characterizations of strong measure zero with a limsup replaced for a liminf.

\begin{definition} 
    Let $\mathrm{\Delta}\in \mathrm{\Gamma}$. A set $X\subseteq \mathbf{C}$ has \emph{$\mathrm{\Delta}$-strong packing dimension zero} if there is a $\mathrm{\Delta}$-computable functional 
    \[F:(\{0,1\}^*\to \mathbb{Q}\cap[1,2]) \rightarrow (\{0,1\}^*\to[0,\infty))\]  
     such that for every acceptable odds function ${O:\{0,1\}^*\to\mathbb{Q}\cap[1,2]}$, $F(O)$ is an $O$-supermartingale with 
     \[\liminf_{n\to \infty} d(S[0\dots n])=\infty\]
    for all $S\in X$.
\end{definition}
For $\mathrm{\Delta}\in \mathrm{\Gamma}_R$ it's easy to see that the change does not affect the sets, so we will focus on the computable and algorithmic versions. There exists another formulation of strong packing dimension zero by Zindulka \cite{zindulka2012small} using box dimensions in the classical setting. He proved a characterization in terms of coverings that inspires the following definition where they can be shown equivalent for $\mathrm{\Delta}$=all. Compared to the earlier coverings, we will now require that each sequence is covered infinitely often and at a certain frequency.

\begin{definition}\label{SPDZD}
    Let $\mathrm{\Delta}\in \mathrm{\Gamma}_{RCA}$, $A=\{a_0,a_1,a_2, \dots\}\subseteq \mathbb{N}$ be an infinite set and $X\subseteq \mathbf{C}$. A function \\${f^A:\mathbb{N}\to \{0,1\}^*}$ is a \emph{frequent $\mathrm{\Delta}$-$A$ covering of $X$} if $f\in \mathrm{\Delta}$ with oracle $A:\mathbb{N}\to \{0,1\}^*$ satisfies the following.
    \begin{enumerate}
        \item if $n\in A$ then $f^A(n)=w$ for some $w$ with $|w|=n$.
        \item If $A_0=\{a_0,\}, A_1=\{a_1,a_2\}, \dots A_i=\{a_{\frac{i(i+1)}{2}},a_{\frac{i(i+1)}{2}+1}, \dots a_{\frac{i(i+1)}{2}+i}\}$ then
        for every $S\in X$ there exits only finitely many $i$ where
        \[f^A(a_n)\not \sqsubseteq S\]
        for all $a_n\in A_i$.
    \end{enumerate}
    We say that $X$ is \emph{$\mathrm{\Delta}$-frequently coverable} if there is a functional $F\in \mathrm{\Delta}$ such that $F(A)$ is a frequent $\mathrm{\Delta}$-$A$ covering of $X$ for all infinite $A\subseteq\mathbb{N}$.

     Again at the algorithmic level $\mathrm{\Delta}$, the above definitions hold for $f^A$ being a partial recursive function that can also be undefined in condition 1 above. Specifically, $f$ corresponds to a oracle Turing machine that outputs a string of length $n$ or does not halt on every $n\in A$.
\end{definition}

\begin{lemma}
    A set $X$ has algorithmic (computable) strong packing dimension zero if and only if it is  algorithmically (computably) frequently coverable.
\end{lemma}

\begin{proof}
    First suppose $X$ has algorithmic (computable) strong packing dimension zero and let $F$ be a odds functional witnessing this. Without loss of generality suppose $F^O(\lambda)<1$ for all odds functions $O$. We will define a Turing machine $M$ that works as follows on input $A=\{a_0,a_1,a_2,\dots\}\subseteq \mathbb{N}$ in standard order. Let $f_A:\mathbb{N}\rightarrow\mathbb{N}$ be the function 
    \[f_A(n)=a_{\frac{n(n+1)}{2}+n}\]
    and consider the odds function
    \[
    O(w)= 
\begin{cases}
    \frac{n+2}{n+1},& \text{if } |w|=f_A(n)\\
    1,              & \text{otherwise}
\end{cases}
\]
Then let $M^A$ be the following algorithm.
\begin{algorithm}[H]
    \caption{}
    \begin{algorithmic}[1]
    \STATE On input $a_i\in A$:
    \STATE Compute the minimum $n\in \mathbb{N}$ such that $a_i\in A_n$.
    \STATE Let $a_i$ be the $c^{th}$ element in $A_n$.
    \STATE Run $F^O$ and for each $w$ with $|w|=f_A(n)$ where $F^O(w)> 1$ is found set $c=c-1$, \STATE Output $w[0\dots a_i]$ when $c$ gets set to zero in step 4. If $\mathrm{\Delta}$=comp then output $0^{a_i}$ if every $w$ with $|w|=f_A(n)$ has been checked and $c>0$.
    \end{algorithmic}
    \end{algorithm}
Note that by construction the maximum number of strings of $w$ with ${|w|=f_A(n)}$ such that $F^O(w)>1$ is at most 
\[(\prod_{i=0}^n\frac{i+2}{i+1})-1=n+1\]
 in order for the supermartingale condition $\eqref{smc}$ to be satisfied. Therefore each of these $n+1$ strings can be output for the set $A_i$. Thus, for every $S\in X$, $F^O(S[0...n])>1$ for all but finitely many $n\in \mathrm{Range}(f_A)$ so it must have a prefix in all but finitely many $A_i$.

For the other direction let $M$ be an oracle Turing machine showing that $X$ is $\mathrm{\Delta}$-frequently coverable. We will define a odds functional that succeeds on $X$ given odds function $O$ as follows. Let 
\[g_O(n)=\max\{\prod_{i=0}^{n-1}O(w[0\dots i])\mid  w\in \{0,1\}^n\}\]
be the maximum product of odds along any string of length $n$ (with $g_O(0)=1$). Define a function $f_O:\mathbb{N}\to \mathbb{N}$ by
\[f_O(n)=\min\{m\in \mathbb{N}\mid \prod_{i=0 }^{m-1}O(w[0\dots i])\geq 2(n+1)\cdot g_O(f_O(n-1)) \text{ for all } w\in \{0,1\}^m\}\]
using $f_O(-1)=0$. This is the minimum length $m$ so that the product of odds along every path of length $m$ is at least $2(n+1)$ times the amount of any string of length $f_O(n-1)$. Note this is possible by compactness. Let 
\[A=\bigcup_{n\in \mathbb{N}} \{f_O(n), f_O(n)+1,\dots f_O(n)+n\}\]
and $A_0,A_1,A_2\dots$ be as stated in the lemma. We will define a $O$-supermartingale $d$ in steps where $d_i$ is the $O$-supermartingale after step $i$ and $d=\lim_{n\to \infty}d_i$. Starting with $d_{-1}(w)=0$ for all $w$, let $d_{i+1}=d_i$ with the following changes.

\begin{itemize}
    \item For each each $a_j\in A_{i+1}$ where $M^A(a_j)=w$:
    \begin{itemize}
        \item Add $\frac{1}{(i+2)2^{(i+2)}}$ to $d_{i+1}(\lambda)$. 
        \item Add $\frac{1}{(i+2)2^{(i+2)}}  \prod_{j=0}^{k}O(w[0\dots j])$ to $d_{i+1}(w[0\dots j])$ \\
        for $0\leq k < f_O(i+1)$.
        \item For each $x$ where $x=M^A(a_k)$ for some $a_k\in A_{i}$ and $x\sqsubset w$:
        \begin{itemize}
            \item add $\frac{1}{i+1}d_{i}(x) \prod_{j=f_O(i)}^{k}O(w[0\dots j])$
    to $d_{i+1}(w[0\dots k])$ \\for all ${f_O(i)\leq k < f_O(i+1)}$.
        \end{itemize}
    \end{itemize}
\end{itemize}
Note that there are at most $i+1$ extensions of a $x$ in $A_i$ to a $w$ in $A_{i+1}$ so this results in a $O$-supermartingale.
At the computable level, $d$ can be computed to arbitrary precision by computing $d_i$ exactly after $i$ is large enough. At the algorithmic level, an algorithm can keep track of the seen outputs and perform the necessary changes from last step when connections are found across different levels.

Then we have that
\[d(\lambda)\leq \sum_{i\in \mathbb{N}}(i+1)(\frac{1}{(i+1)2^{(i+1)}})=1\] 
so $F(O)=d$ is an $O$-supermartingale.
 Now suppose that $S\in X$ and let $n$ be such that $A_i$ contains a prefix of $S$ for all $i\geq n$. Let $r>0$ be such that $d(S[0\dots f_O(n)])=r$. Then by construction, for all $k\geq f_O(n+i)$ we have 
\[d(S[0\dots k])\geq \frac{2^ir}{n+i+1}\] 
and hence 
\[\liminf_{n\to \infty}d^O(S[0\dots n])=\infty.\]
\end{proof}

\begin{lemma}\label{trees}
    If a sequence $S$ has algorithmic strong packing dimension zero then there is a fixed constant $c\in \mathbb{N}$ and a Turing Machine $M$ such that $M$ halts on at most $c$ inputs of length $n$, one of which is $S[0\dots n]$, for all $n\in \mathbb{N}$.
    \end{lemma}

\begin{proof}
    We will prove the contrapositive. Assume no such $M$ exists and let $N$ be any oracle Turing machine. We will construct an $A\subseteq \mathbb{N}$ such that $N^A$ does not produce a frequent cover. 

   We do this in stages. At stage $s$, let $A^{s-1}$ be the finite set of naturals in $A$ after stage $s-1$, initially empty, and let $m_{s-1}=\max\{A^{s-1}\}$. Let $n_{s}$ be the amount of naturals needed to finish the next $A_i$ along with $A_{i-1}$ if necessary as defined in the third condition of Definition \ref{SPDZD}. Now let $m\in \mathbb{N}>m_{s-1}$ be such that the last condition in the following construction holds:
    \begin{itemize}
        \item Let $A'=A^{s-1}\cup\{m, m+1, \dots m+n_s-1\}$.
        \item Dovetail $N$ on inputs $m,m+1, \dots, m+n_s-1$ and all finite oracles $A$ that extend $A'$.
        \item Whenever $N$ halts on one of these inputs update $A'$ to be this new finite oracle that caused $N$ to halt and restart dovetailing on the rest.
        \item After $N$ has halted on all inputs or $A'$ has been defined to where no extension of $A'$ will cause any of the remaining inputs to halt, none of the outputs are a prefix of $S$. 
    \end{itemize}
    Then define $A_{s}$ to be the final $A'$ above and go to the next stage. 
    
    It is clear $A_i$ will not contain a prefix of $S$ so it suffices to show that there is such an $m$ at each stage. If this was not possible at some stage $s$, then using the finite $A^{s-1}$ and the machine $N$, it is possible to create a new Turing machine $M$ that performs the above dovetailing for each $m$ and creates a computably enumerable set of size at most $n_s=c$ strings of length m with one being a prefix of $S$, contradicting our assumption. 
\end{proof}

\begin{corollary}
        If a sequence $S$ has algorithmic or computable strong packing dimension zero then $S$ is decidable.
\end{corollary}
\begin{proof}
    By Lemma \ref{trees}, any $S$ with algorithmic strong packing dimension zero has $\mathrm{K}(S[0\dots n]|n)\leq c$ for some $c$ and all $n$. Therefore $S$ is decidable, see for instance \cite{LiVit19}. 
\end{proof}

However, this is not enough to say every computable strong packing dimension zero set is computably countable. For example, the set DEC of all decidable languages is not computably countable. We will see that computable strong packing dimension zero sets coincide with the following notion.

\begin{definition}
    A \emph{weak constructor} is a function $\delta:\mathbb{N}\rightarrow ({\{0,1\}^*})^b$ for some $b\in\mathbb{N}$ satisfying
    \[w\in \delta(n) \Rightarrow \exists u \in \delta(n-1) \text{ with } u\sqsubset w \]
    for all $n\in \mathbb{N}^+$.
    The \emph{result} of a weak constructor is the set 
    \[\{S\in \mathbf{C} \mid \forall n \in \mathbb{N} \enspace  \exists w \in \delta(n) \enspace w\sqsubseteq S\}.\]
\end{definition}

It is easy to see that every normal constructor coincides with a weak constructor with $b=1$ and that a language is decidable if and only if has a computable weak constructor. One can view a weak constructor as a growing tree with at most $b$ ``alive'' branches at any stage. The extra power comes from not having to decide which branch to follow in a computable amount of time when looking at effective unions of constructors.

\begin{definition}
    A set $X\subseteq \mathbf{C}$ is weakly $\mathrm{\Delta}$-countable if there exists a function $\delta:\mathbb{N} \times \mathbb{N}\rightarrow \mathcal{P}\{0,1\}^*$ meeting the following properties.
    \begin{enumerate}
        \item $\delta\in \mathrm{\Delta}$.
        \item For each $k\in \mathbb{N}$, if we write $\delta_k(n)=\delta(k,n)$, then the function $\delta_k$ is a weak constructor.
        \item $X\subseteq \bigcup_{k\in \mathbb{N}} R(\delta_k)$.
    \end{enumerate}
\end{definition}

\begin{example}
    Consider an enumeration $M_0,M_1, M_2,\dots$ of all Turing machines and for each $M_i$ let 
    \[H_i=\{n\in \mathbb{N} \mid M_i(n) \text{ halts within } n^2 \text{ steps}\}.\]
    Now remove every other element of $H_i$ (the odd indices) to get the set \\${H'_i=\{n_0,n_1,n_2,\dots\}}$ and create a language 
    \[L_i=\cup_i\{n_i,n_i+1 \mid n_{i+1} \text{ is odd}\} \bigcup\cup_i\{n_i\mid n_{i+1} \text{ is odd}\}.\]
    Finally, let
    \[X=\cup_{i\in \mathbb{N}}L_i.\]
    Then $X$ is weakly computably countable since it can guess at each step in $L_i$ that there are no more inputs on $M_i$ that halt in $n^2$ steps, that the next is even, or that it is odd. However, for a normal constructor to succeed in a similar way it would have to know if there is another input that halts in $n^2$ steps to avoid running forever without producing output.
\end{example}
Now note that Lemma \ref{trees} can be extended to computability as follows.
\begin{lemma}\label{comptrees}
    If a set $X$ has computable strong packing dimension zero then there is a total Turing Machine $M$ with the following property. For each $S\in X$, there is a string $w_{S}$ and  fixed constant $c_S\in \mathbb{N}$ such that $M(w_S,x)$ halts on at most $c_S$ strings $x$ of length $n$, one of which is $S[0\dots n]$, for all $n\in \mathbb{N}$.
\end{lemma}
\begin{proof}
    The proof is the same as Lemma \ref{trees} except we only care about total Turing machines and pick an $S$ where an $m$ can be chosen at each stage. If that were not possible, then there would be a computable $M$ that works for each $S$ using a finite amount of information encoded in a string $w_S$ along with the total Turing machine $N$ defined in the proof. 
\end{proof}
\begin{theorem}
    A set $X$ has computable strong packing dimension zero if and only if it is weakly computably countable.
\end{theorem}
\begin{proof}
    First, suppose $X$ has computable strong packing dimension zero. Then let $M$ be a total Turing machine as in Lemma \ref{comptrees}. Let $\delta:\mathbb{N} \times \mathbb{N}\rightarrow \mathcal{P}\{0,1\}^*$  be defined as follows. If $k=(w,c)$ for some $w\in\{0,1\}^*$ and $c\in\mathbb{N}$, then let $\delta_k(n)$ output the set of strings $x$ of length $n$ where $M(w,x)$ halts. If this is ever more than $c$ or if $k\neq(w,c)$ for some $(w,c)$ then let  $\delta_k(n)=\{0^n\}$ from then on. Then each $\delta_k$ is a weak constructor. For all $S\in X$, when $k=(w_S,c_S)$ we have $S\in R(\delta_k)$ so $X$ is weakly computably countable.

    For the other direction let $X$ be weakly computably countable via $\delta$. Then an odds functional $F$ can be created as follows given odds function $O$. Let $d_{k,c}$ be a $O$-supermartingale with $d_{k,c}(\lambda)=2^{k-c}$. Then the values of $d_{k,c}$ are generated using the following computable recursive algorithm starting with $w=\lambda$.

\begin{algorithm}[H]
\caption{$d_{k,c}$}
\begin{algorithmic}[1]
\STATE On input $w$:
\STATE Compute $n\in \mathbb{N}$ where $\prod_{i=0}^{|x|-1}O(x[0\dots i])\geq 2c\prod_{i=0}^{|w|-1}O(w[0\dots i])$ for all $x\in \delta_k(n)$. 
\STATE If $|\delta_k(n)|>c$ stop/let all undefined $d_{k,c}(x)=0$. 
\STATE If $\delta_k(n)$ is not prefix free remove any extensions of elements in $\delta_k(n)$ to make it prefix free.
\STATE For each $x \in \delta_k(n)$ with $w\sqsubset  x$, let ${d_{k,c}(x[0\dots i])=\frac{1}{c}d_{k,c}(w)\prod_{j=|w|}^{i} O(x[0\dots j])}$ for all $|w|\leq i<|x|$.
\STATE Start step 1 with $w=x$ for each such $x$.
\end{algorithmic}
\end{algorithm}
Then once all values of $\delta_k(n)$ are greater than the length of a string $w$, the value of $d_{k,c}(w)$ is fixed and computable.
    
Moreover, for the constant maximal width $b$ of the weak constructor $\delta_k$, we have that 
    \[\liminf_{n\to \infty}d_{k,b}(S[0\dots n])=\infty\] 
    for all $S\in R(\delta_k)$. Lastly, 
    letting \[F^O=\sum_{k,c\in \mathbb{N}} d_{k,c}\] 
    we have $F^O(\lambda)\leq \sum_{k\in \mathbb{N}}\sum_{c\in\mathbb{N}} 2^{-k-c}=4$. 
\end{proof}

\begin{remark}
    Franklin et al. \cite{franklin2013anti} defined a sequence $S$ to be \emph{anti-complex} if for every computable order (non-decreasing and unbounded) $f:\mathbb{N}\to \mathbb{N}$, it is the case that $\mathrm{K}(S[0\dots f(n)])\leq n $ for almost every $n$. Hölzl and Porter \cite{holzl2017randomness} showed this is equivalent to $\mathrm{KM}(S[0\dots n])\leq f(n)$ for almost every $n$. In a version of our strong algorithmic packing dimension zero that is changed to only require success with respect to computable odds functions (or equivalently computable prediction orders or computable gauge functions), it is routine to see that this gives an equivalent characterization of anti-complex sequences by an analogue of Lemma \ref{algsmz}. Similarly, a version of algorithmic strong measure zero only requiring success on computable objects corresponds to  \emph{i.o. anti-complex} \cite{holzl2017randomness} sequences where the above characterizations of anti-complex only need to hold for infinitely many $n$.
\end{remark}

\section*{Acknowledgments}
The author thanks Jack Lutz for many helpful conversations.

\bibliography{EVSMZ}
\end{document}